\documentclass[a4paper,12pt]{article}
\usepackage{amssymb,amsmath,amsthm}
\usepackage{amsfonts}
\usepackage[english]{babel}
\usepackage{hyperref}
\usepackage{color}
\usepackage{xcolor}

\numberwithin{equation}{section}

\newtheorem{theorem}{Theorem}[section]
\newtheorem{proposition}[theorem]{Proposition}
\newtheorem{lemma}[theorem]{Lemma}

\newtheorem{remark}[theorem]{Remark}
\newtheorem{remarks}[theorem]{Remark}
\newtheorem{definition}[theorem]{Definition}
\newcommand{\be}{\begin{equation}}
\newcommand{\ee}{\end{equation}}

\newcommand{\e}{\varepsilon}

\newcommand{\R}{\mathbb R}

\newcommand{\Z}{\mathbb Z}

\newcommand{\N}{\mathbb N}
\newcommand{\T}{\mathbb T}

\newcommand{\ii }{{\rm i} }

\newcommand{\vphi}{\varphi }

\def\ba{\begin{aligned}}
\def\ea{\end{aligned}}
\def\beginm{\begin{multline}}
\def\endm{\end{multline}}

\let\div\undefined
\DeclareMathOperator{\div}{div}

\begin{document}

\title{{\bf The Navier-Stokes equation with time quasi-periodic external force: existence and stability of quasi-periodic solutions}}

\date{}

\author{Riccardo Montalto}

\maketitle

\noindent
{\bf Abstract.}
We prove the existence of small amplitude, time-quasi-periodic solutions (invariant tori)
for the incompressible Navier-Stokes equation 
on the $d$-dimensional torus $\T^d$, 
with a small, quasi-periodic in time
external force. We also show that they are orbitally and asymptotically stable in $H^s$ (for $s$ large enough). More precisely, for any initial datum which is close to the invariant torus, there exists a unique global in time solution which stays close to the invariant torus for all times. Moreover, the solution converges asymptotically to the invariant torus for $t \to + \infty$, with an exponential rate of convergence $O( e^{- \alpha t })$ for any arbitrary $\alpha \in (0, 1)$. 

\smallskip 





\tableofcontents

\section{Introduction}\label{introduction}

We consider the Navier-Stokes equation for an incompressible fluid on the $d$-dimensional torus $\T^d$, $d \geq 2$, $\T := \R / 2 \pi \Z$,
\begin{equation}\label{Eulero1}
\begin{cases}
\partial_t u - \Delta u + u \cdot \nabla u  + \nabla p = \e f(\omega t, x) \\
\div u = 0
\end{cases} 
\end{equation}
where $\e \in (0, 1)$ is a small parameter, the frequency $\omega = (\omega_1, \ldots, \omega_\nu) \in \R^\nu$ is a $\nu$-dimensional vector and $f : \T^\nu \times \T^d \to \R^d$ is a smooth quasi-periodic external force. The unknowns of the problem are the velocity field $u = (u_1, \ldots, u_d) : \R \times \T^d \to \R^d$, 
and the pressure $p : \R \times \T^d \to \R$. For convenience, we set the viscosity parameter in front of the laplacian equal to one. We assume that $f$ has zero space-time average, namely
\begin{equation}\label{condizioni media f}
\int_{ \T^\nu \times \T^d} f(\vphi, x)\, d \vphi \, d x = 0\,. 
\end{equation}
The purpose of the present paper is to show the existence and the stability of smooth quasi-periodic solutions of the equation \eqref{Eulero1}. More precisely we show that if $f$ is a sufficiently regular vector field satisfying \eqref{condizioni media f}, for $\e$ sufficiently small and for $\omega \in \R^\nu$ diophantine, i.e. 
\begin{equation}\label{def diofanteo}
\begin{aligned}
& |\omega \cdot \ell| \geq \frac{\gamma}{|\ell|^\nu}, \quad \forall \ell \in \Z^\nu \setminus \{ 0 \}\,, \\
& \text{for some}\qquad \qquad  \gamma \in (0, 1), 
\end{aligned}
\end{equation}
\footnote{It is well known that a.e. frequency in $\R^\nu$ (w.r. to the Lebesgue measure) is diophantine.} then the equation \eqref{Eulero1} admits smooth quasi-periodic solutions (which are referred to also as invariant tori) $u_\omega (t, x) = U(\omega t, x)$, $p_\omega (t, x) = P(\omega t, x)$, $U : \T^\nu \times \T^d \to \R^d$, $P : \T^\nu \times \T^d \to \R$ of size $O(\e)$, oscillating with the same frequency $\omega \in \R^\nu$ of the forcing term. If the forcing term has zero-average in $x$, i.e. 
 \begin{equation}\label{condizioni media f 2}
 \int_{\T^d} f(\vphi, x)\, d x= 0, \quad \forall \vphi \in \T^\nu
 \end{equation}
 then the result holds for any frequency vector $\omega \in \R^\nu$, without requiring any non-resonance condition. Furthermore, we show also the orbital and the asymptotic stability of these quasi-periodic solutions in high Sobolev norms. More precisely, for any sufficiently regular initial datum which is $\delta$-close to the invariant torus (w.r. to the $H^s$ topology), the corresponding solution of \eqref{Eulero1} is global in time and it satisfies the following properties. 
 \begin{itemize}
 \item{\bf Orbital stability.} For all times $t \geq 0$, the distance in $H^s$ between the solution and the invariant torus is of order $O(\delta)$.
 \item{\bf Asymptotic stability.} The solution converges asymptotically to the invariant torus in high Sobolev norm $\| \cdot \|_{H^s_x}$ as $t \to + \infty$, with a rate of convergence which is exponential, i.e. $O(e^{- \alpha t})$, for any arbitrary $\alpha \in (0, 1)$.
 \end{itemize} 
 
 \noindent
In order to state precisely our main results, we introduce some notations. For any vector $a = (a_1, \ldots, a_p) \in \R^p$, we denote by $|a|$ its Euclidean norm, namely $|a| := \sqrt{a_1^2 + \ldots + a_p^2}$. Let $d, n \in \N$ and a function $u \in L^2(\T^d, \R^n)$. Then $u(x)$ can be expanded in Fourier series 
$$
u(x) = \sum_{\xi \in \Z^d} \widehat u(\xi) e^{\ii x \cdot \xi}
$$ 
where its Fourier coefficients $\widehat u(\xi)$ are defined by 
$$
\widehat u(\xi) := \frac{1}{(2 \pi)^d} \int_{\T^d} u(x) e^{- \ii x \cdot \xi}\,d x, \quad \forall \xi \in \Z^d\,. 
$$ 
For any $s \geq 0$, we denote by $H^s(\T^d, \R^n)$ the standard Sobolev space of functions $u : \T^d \to \R^n$ equipped by the norm 
\begin{equation}\label{norma sobolev 0}
\| u \|_{H^s_x} := \Big(\sum_{\xi \in \Z^d} \langle \xi \rangle^{2 s} |\widehat u(\xi)|^2 \Big)^{\frac12}, \quad \langle \xi \rangle := {\rm max}\{ 1, |\xi| \}\,. 
\end{equation}
We also define the Sobolev space of functions with zero average 
\begin{equation}\label{sobolev media nulla}
H^s_0(\T^d, \R^n) := \Big\{ u \in H^s(\T^d, \R^n) : \int_{\T^d} u(x)\, d x = 0 \Big\}\,. 
\end{equation}
Moreover, given a Banach space $(X, \| \cdot \|_X)$ and an interval ${\cal I} \subseteq \R$, we denote by ${\cal C}^0_b({\cal I}, X)$ the space of bounded, continuous functions $u : {\cal I} \to X$, equipped with the sup-norm 
$$
\| u \|_{{\cal C}^0_t X} := \sup_{t \in {\cal I}} \| u(t) \|_X\,. 
$$
For any integer $k \geq 1$, ${\cal C}^k_b({\cal I}, X)$ is the space of $k$-times differentiable functions $u : {\cal I} \to X$ with continuous and bounded derivatives equipped with the norm 
$$
\| u \|_{{\cal C}^k_t X} := {\rm max}_{n \leq k} \| \partial_t^n u \|_{{\cal C}^0_t X}\,. 
$$
In a similar way we define the spaces ${\cal C}^0(\T^\nu, X)$, ${\cal C}^k(\T^\nu, X)$, $k \geq 1$ and the corresponding norms $\| \cdot \|_{{\cal C}^0_\vphi X}$, $\| \cdot \|_{{\cal C}^k_\vphi X}$ (where $\T^\nu$ is the $\nu$-dimensional torus). We also denote by ${\cal C}^N(\T^\nu \times \T^d, \R^d)$ the space of $N$-times continuously differentiable functions $\T^\nu \times \T^d \to \R^d$ equipped with the standard ${\cal C}^N$ norm $\| \cdot \|_{{\cal C}^N}$. 

\noindent
{\bf Notation.} Throughout the whole paper, the notation $A \lesssim B$ means that there exists a constant $C$ which can depend on the number of frequencies $\nu$, the dimension of the torus $d$, the constant $\gamma$ appearing in the diophantine condition \eqref{def diofanteo} and on the ${\cal C}^N$ norm of the forcing term $\| f \|_{{\cal C}^N}$. Given $n$ positive real numbers $s_1, \ldots, s_n > 0$, we write $A \lesssim_{s_1, \ldots, s_n} B$ if there exists a constant $C = C(s_1, \ldots, s_n) > 0$ (eventually depending also on $d, \nu, \gamma, \| f \|_{{\cal C}^N}$) such that $A \leq C B$. 

\medskip

\noindent
We are now ready to state the main results of our paper. 
\begin{theorem}[\bf Existence of quasi-periodic solutions]\label{esistenza quasi-periodiche}
Let $s > d/2 + 1$, $N > \frac{3 \nu}{2} + s + 2$, $\omega \in \R^\nu$ diophantine (see \eqref{def diofanteo}) and assume that the forcing term $f$ is in $ {\cal C}^N(\T^\nu \times \T^d, \R^d)$ and it satisfies \eqref{condizioni media f}. Then there exists $\e_0 = \e_0(f, s, d, \nu ) \in (0,  1)$ small enough and a constant $C = C(f,  s, d, \nu) > 0$ large enough such that for any $\e \in (0, \e_0)$ there exist  $U \in {\cal C}^1(\T^\nu, H^s(\T^d, \R^d))$, $P \in {\cal C}^0(\T^\nu, H^s(\T^d, \R))$ satisfying
$$
\int_{\T^\nu \times \T^d} U(\vphi, x)\, d \vphi\, d x = 0, \quad \int_{\T^d} P(\vphi, x)\, d x= 0, \quad \forall \vphi \in \T^\nu
$$ such that $(u_\omega(t, x), p_\omega( t, x)) : = (U(\omega t, x), P(\omega t, x))$ solves the Navier-Stokes equation \eqref{Eulero1} and $$\| U \|_{{\cal C}^1_\vphi H^s_x}, \| P \|_{{\cal C}^0_\vphi H^s_x} \leq C \e\,.$$If the forcing term $f$ has zero space average, i.e. it satisfies \eqref{condizioni media f 2}, then the same statement holds for any frequency vector $\omega \in \R^\nu$ and $U(\vphi, x)$ satisfies 
$$
\int_{\T^d} U(\vphi, x)\, d x = 0, \quad \forall \vphi \in \T^\nu\,. 
$$
\end{theorem}

\begin{theorem}[\bf Stability]\label{stabilita asintotica}
Let $\alpha \in (0, 1)$, $s > d/2 + 1$, $N > \frac{3 \nu}{2} + s + 2$, $u_\omega$, $p_\omega$ be given in Theorem \ref{esistenza quasi-periodiche}. Then there exists $\delta = \delta(f , s, \alpha,  d, \nu) \in (0, 1)$ small enough and a constant $C = C(f, s, \alpha, d, \nu) > 0$ large enough such that for $\e \leq \delta$ and for any initial datum $u_0 \in H^s(\T^d, \R^d)$ satisfying
$$
\| u_0 - u_\omega(0, \cdot) \|_{H^s_x} \leq \delta, \quad \int_{\T^d} \Big( u_0(x) - u_\omega(0, x) \Big)\, d x = 0
$$
 there exists a unique global classical solution $(u, p)$ of the Navier-Stokes equation \eqref{Eulero1} with initial datum $u(0, x) = u_0(x)$ which satisfies 
$$
\begin{aligned}
& u \in {\cal C}^0_b \Big([0, + \infty), H^s(\T^d, \R^d) \Big) \cap {\cal C}^1_b \Big([0, + \infty), H^{s - 2}(\T^d, \R^d) \Big)\,, \quad p \in {\cal C}^0_b \Big([0, + \infty), H^s_0(\T^d, \R) \Big), \\
& \int_{\T^d}\Big( u(t, x) - u_\omega (t, x) \Big)\, d x = 0\,, \quad \forall t \geq 0\,, \\
& \| u(t, \cdot ) - u_\omega (t, \cdot) \|_{H^s_x}\,,\, \| \partial_t u(t, \cdot ) - \partial_t u_\omega (t, \cdot) \|_{H^{s - 2}_x}\,,\, \| p(t, \cdot) - p_\omega (t, \cdot) \|_{H^s_x} \leq C \delta e^{- \alpha t}
\end{aligned}
$$
for any $t \geq 0$. 
\end{theorem}
The investigation of the Navier-Stokes equation with time periodic external force dates back to Serrin \cite{serrin}, Yudovich \cite{Yu 2}, Lions \cite{Lions}, Prodi \cite{prodi2} and Prouse \cite{prouse1}. In these papers the authors proved the existence of weak periodic solutions on bounded domains, oscillating with the same frequency of the external force. The existence of weak quasi-periodic solutions in dimension two has been proved by Prouse \cite{prouse2}. More recently these results have been extended to unbounded domains by Maremonti \cite{maremonti}, Maremonti-Padula \cite{maremonti 2}, Salvi \cite{Salvi} and then by Galdi \cite{Galdi1}, \cite{Galdi2}, Galdi-Silvestre \cite{Galdi3}, Galdi-Kyed \cite{Galdi4} and Kyed \cite{Kyed}. We point out that in some of the aforementioned results, no smallness assumptions on the forcing term are needed and therefore, the periodic solutions obtained are not small in size, see for instance \cite{serrin}, \cite{Yu 2}, \cite{prodi2}, \cite{prouse1}, \cite{prouse2}, \cite{maremonti 2}, \cite{Salvi}. The asymptotic stability of periodic solutions (also referred to as {\it attainability property}) has been also investigated in \cite{maremonti}, \cite{maremonti 2}, but it is only proved w.r. to the $L^2$-norm and the rate of convergence provided is $O(t^{- \eta})$ for some constant $\eta > 0$. More recently Galdi and Hishida \cite{Galdi5} proved the asymptotic stability for the Navier-Stokes equation with a traslation velocity term, by using the Lorentz spaces and they provided a rate of convergence which is essentially $O(t^{- \frac12 + \e})$. In the present paper we consider the Navier-Stokes equation on the $d$-dimensional torus with a small, quasi-periodic in time external force. We show the existence of smooth quasi-periodic solutions (which are also referred to as invariant tori) of small amplitude and we prove their orbital and asymptotic stability in $H^s$ for $s$ large enough (at least larger than $d/2 + 1$). Furthermore the rate of convergence to the invariant torus, in $H^s$, for $t \to + \infty$ is of order $O(e^{- \alpha t})$ for any arbitrary $\alpha \in (0, 1)$. To the best of our knowledge, this is the first result of this kind. 

\noindent
It is also worth to mention that the existence of quasi-periodic solutions, that is also referred to as KAM (Kolmogorov-Arnold-Moser) theory, for dispersive and hyperbolic-type PDEs is a more difficult matter, due to the presence of the so-called {\it small divisors problem}. The existence of time-periodic and quasi-periodic solutions of PDEs started in the late 1980s with the pioneering papers of Kuksin \cite{K87}, Wayne \cite{Wayne} and Craig-Wayne \cite{CW}, see also \cite{K2-KdV}, \cite{LY} for generalizations to PDEs with unbounded nonlinearities. We refer to the recent review \cite{Berti-BUMI-2016} for a complete list of references.

\noindent
Many PDEs arising from fluid dynamics like the water waves equations or the Euler equation are fully nonlinear or quasi-linear equations  (the nonlinear part contains as many derivatives 
as the linear part). 
The breakthrough idea, based on pseudo-differential calculus and micro-local analysis, in order to deal with these kind of PDEs has been introduced 
by Iooss, Plotnikov and Toland \cite{IPT} 
in the problem of finding periodic solutions for the water waves equation. 
The methods developed in \cite{IPT}, combined with a {\it KAM-normal form} procedure 
have been used to develop a general method for PDEs in one dimension,
which allows to construct \emph{quasi-periodic} solutions 
of quasilinear and fully nonlinear PDEs, see \cite{BBM-Airy}, \cite{FP}, \cite{Berti-Montalto}, \cite{BBHM} and references therein.
The extension of KAM theory to higher space dimension $d > 1$
is a difficult matter due to the presence of very strong resonance-phenomena, often related to high multiplicity of eigenvalues. The fist breakthrough results in this directions (for equations with perturbations which do not contain derivatives) have been obtained by Eliasson and Kuksin \cite{EK} and by Bourgain \cite{B} (see also Berti-Bolle \cite{BB1}, \cite{BB2}, Geng-Xu-You \cite{Geng}, Procesi-Procesi \cite{PP}, Berti-Corsi-Procesi \cite{BCP}.) 

Extending KAM theory to PDEs with unbounded perturbations in higher space dimension
is one of the main 
open problems in the field.
Up to now, this has been achieved only in few examples, see \cite{Mon}, \cite{CorsiMontalto}, \cite{FGMP}, \cite{BLM}, \cite{BGMR2}, \cite{BGMR1} and recently on the 3D Euler equation \cite{Eulero} which is the most meaningful physical example. 

\noindent
For the Navier-Stokes equation, unlike in the aforementioned papers on KAM for PDEs, the existence of quasi-periodic solutions is not a small divisors problem and it can be done by using a classical fixed point argument. This is due to the fact that the Navier-Stokes equation is a parabolic PDE and the presence of dissipation avoids the small divisors. In the same spirit, it is also worth to mention \cite{Corsi}, in which the authors investigate quasi-periodic solutions of nonlinear wave equations with dissipation. We also point out that the present paper is the first example in which the stability of invariant tori, in high Sobolev norms, is proved for all times (and it is even an asymptotic stability). This is possible since the presence of the dissipation allows to prove strong time-decay estimates from which one deduces  orbital and asymptotic stability. In the framework of dispersive and hyperbolic PDEs, the orbital stability of invariant tori is usually proved only for {\it large times} by using normal form techniques. The first result in this direction has been proved in \cite{Mem cinesi}. In the remaining part of the introduction, we sketch the main points of our proof.

\noindent 
 As we already explained above, the absence of small divisors is due to the fact that the Navier-Stokes equation is a parabolic PDE. More precisely, this fact is related to invertibility properties of the linear operator $L_\omega := \omega \cdot \partial_\vphi - \Delta$ (where $\omega \cdot \partial_\vphi := \sum_{i = 0}^\nu \omega_i \partial_{\vphi_i}$) acting on Sobolev spaces of functions $u(\vphi, x)$, $(\vphi, x) \in \T^\nu \times \T^d$ with zero average w.r. to $x$. Since the eigenvalues of $L_\omega$ are $\ii \omega \cdot \ell + |j|^2$, $\ell \in \Z^\nu$, $j \in \Z^d \setminus \{ 0 \}$, the inverse of $L_\omega$ {\it gains two space derivatives}, see Lemma \ref{invertibilita L omega}. This is suffcient to perform a fixed point argument on the map $\Phi$ defined in \eqref{mappa punto fisso} from which one deduces the existence of smooth quasi-periodic solutions of small amplitude. The asymptotic and orbital stability of quasi-periodic solutions (which are constructed in Section \ref{sezione costruzione quasi-periodiche}) are proved in Section \ref{sezione stabilita asintotica}. More precisely we show that for any initial datum $u_0$ which is $\delta$-close to the quasi-periodic solution $u_\omega(0, x)$ in $H^s$ norm (and such that $u_0 - u_\omega(0, \cdot)$ has zero average), there exists a unique classical solution $(u, p)$ such that 
$$
\| u (t, \cdot) - u_\omega(t, \cdot) \|_{H^s_x} =  O(\delta e^{- \alpha t}), \quad \| p (t, \cdot) - p_\omega(t, \cdot) \|_{H^s_x} =  O(\delta e^{- \alpha t})\,, \quad \alpha \in (0, 1) 
$$
 for any $t \geq 0$. This is exactly the content of Theorem \ref{stabilita asintotica}, which easily follows from Proposition \ref{proposizione stabilita asintotica con Leray}. This Proposition is proved also by a fixed point argument on the nonlinear map $\Phi$ defined in \eqref{mappa punto fisso stabilita asintotica} in weighted Sobolev spaces ${\cal E}_s$ (see \eqref{def cal Es}), defined by the norm 
$$
\| u \|_{{\cal E}_s} := \sup_{t \geq 0} e^{\alpha t} \| u(t, \cdot) \|_{H^s_x}
$$
where $\alpha \in (0, 1)$ is a fixed constant. The fixed point argument relies on some {\it dispersive-type estimates} for the heat propagator $e^{t \Delta}$, which are proved in Section \ref{stime dispersive propagatore calore}. The key estimates are the following. 
\begin{enumerate}
\item For any $u_0 \in H^{s - 1}(\T^d, \R^d)$ with zero average and for any $n \in \N$, $\alpha \in (0, 1)$, $t > 0$, one has 
\begin{equation}\label{bla bla bla 0}
\| e^{t \Delta} u_0 \|_{H^s_x} \leq C(n, \alpha) t^{- \frac{n}{2}} e^{- \alpha t} \| u_0 \|_{H^{s - 1}_x}
\end{equation}
for some constant $C(n, \alpha) > 0$ (see Lemma \ref{lemma propagatore libero calore}). This estimate states that the heat propagator {\it gains one-space derivative} and exponential decay in time $e^{- \alpha t} t^{- \frac{n}{2}}$. Note that, without gain of derivatives on $u_0$, the exponential decay is stronger, namely $e^{- t}$, see Lemma \ref{lemma propagatore libero calore}-$(i)$.    
\item For any $f \in {\cal E}_{s - 1}$ 
\begin{equation}\label{bla bla bla 1}
\Big\| \int_0^t e^{(t - \tau)\Delta} f(\tau, \cdot)\, d\tau \Big\|_{H^s_x} \leq C(\alpha) e^{- \alpha t} \| f \|_{{\cal E}_{s - 1}}
\end{equation}
for some constant $C(\alpha) > 0$ (see Proposition \ref{stima dispersiva principale}). This estimate states that the integral term which usually appears in the Duhamel formula (see \eqref{mappa punto fisso stabilita asintotica}) gains one space derivative w.r. to $f (t, x)$ and keeps the same exponential decay in time as $f(t, x)$. 
\end{enumerate}
We also remark that the constants $C(n, \alpha)$, $C(\alpha)$ appearing in the estimates \eqref{bla bla bla 0}, \eqref{bla bla bla 1} tend to $\infty$ when $\alpha \to 1$. This is the reason way it is not possible to get a decay $O(e^{- t})$ in the asymptotic stability estimate provided in Theorem \ref{stabilita asintotica}.

\noindent
The latter two estimates allow to show in Proposition \ref{contrazione punto fisso stabilita asintotica} that the map $\Phi$ defined in \eqref{mappa punto fisso stabilita asintotica} is a contraction. The proof of Theorem \ref{stabilita asintotica} is then easily concluded in Section \ref{fine dim stabilita asintotica}. 

\noindent
As a concluding remark, it is also worth to mention that our methods does not cover the zero viscosity limit $\mu \to 0$, where $\mu$ is the usual viscosity parameter in front of the laplacian (that we set for convenience equal to one). Indeed some constants in our estimates become infinity when $\mu  \to 0$. Actually, it would be very interesting to study the {\it singular perturbation problem} for $\mu \to 0$ and to see if one is able to recover the quasi-periodic solutions of the Euler equation constructed in \cite{Eulero}.

\bigskip

\noindent
{\sc Acknowledgements.} The author is supported by INDAM-GNFM. The author warmly thanks Dario Bambusi, Luca Franzoi, Thomas Kappeler and Alberto Maspero for their feedbacks.  
\section{Functional spaces}
In this section we collect some standard technical tools which will be used in the proof of our results.  

\noindent
For $u = (u_1, \ldots, u_n) \in H^s(\T^d, \R^n)$, 
\begin{equation}\label{equivalenza sobolev 0}
\| u \|_{H^s_x} \simeq {\rm max}_{i = 1, \ldots, n} \| u_i \|_{H^s_x}\,. 
\end{equation}
The following standard algebra lemma holds. 
\begin{lemma}\label{interpolazione sobolev 0}
Let $s > d/2$ and $u, v \in H^s(\T^d, \R^n)$. Then $u \cdot v \in H^s(\T^d, \R)$ (where $\cdot$ denotes the standard scalar product on $\R^n$) and $\| u \cdot v \|_{H^s_x} \lesssim_s \| u \|_{H^s_x} \| v \|_{H^s_x}$. 
\end{lemma}
We also consider functions 
$$
\T^\nu \to L^2(\T^d, \R^n), \quad \vphi \mapsto u(\vphi, \cdot)
$$
which are in $L^2\Big(\T^\nu, L^2(\T^d, \R^n)\Big)$. We can write the Fourier series of a function 

$u \in L^2\Big(\T^\nu, L^2(\T^d, \R^n)\Big)$ as 
\begin{equation}\label{fourier tempo 0}
u(\vphi, \cdot) = \sum_{\ell \in \Z^\nu} \widehat u(\ell, \cdot) e^{\ii \ell \cdot \vphi}
\end{equation}
where 
\begin{equation}\label{fourier tempo 1}
\widehat u(\ell, \cdot) := \frac{1}{(2 \pi)^\nu} \int_{\T^\nu}u(\vphi, \cdot) e^{- \ii \ell \cdot \vphi}\, d \vphi \in L^2(\T^d, \R^n), \quad \ell \in \Z^\nu\,. 
\end{equation}
By expanding also the function $\widehat u(\ell, \cdot)$ in Fourier series, we get 
\begin{equation}\label{fourier tempo 2}
\begin{aligned}
& \widehat u(\ell , x) = \sum_{j \in \Z^d} \widehat u(\ell, j) e^{\ii j \cdot x}\,, \\
& \widehat u(\ell , j) := \frac{1}{(2 \pi)^{\nu + d}} \int_{\T^{\nu + d}} u(\vphi, x) e^{- \ii \ell \cdot \vphi} e^{- \ii j \cdot x}\, d \vphi\, d x, \quad (\ell, j) \in \Z^\nu \times \Z^d
\end{aligned}
\end{equation}
and hence we can write 
\begin{equation}\label{fourier tempo 3}
u(\vphi, x) = \sum_{\ell \in \Z^\nu} \sum_{j \in \Z^d} \widehat u(\ell, j) e^{\ii \ell \cdot \vphi} e^{\ii j \cdot x}\,. 
\end{equation}
For any $\sigma, s \geq 0$, we define the Sobolev space $H^\sigma\Big(\T^\nu, H^s(\T^d, \R^n) \Big)$ as the space of functions $u \in L^2\Big(\T^\nu, L^2(\T^d, \R^n) \Big)$ equipped by the norm 
\begin{equation}\label{norma sobolev mista}
\| u \|_{\sigma, s} \equiv \| u \|_{H^\sigma_\vphi H^s_x} := \Big( \sum_{\ell \in \Z^\nu} \langle \ell \rangle^{2 \sigma} \| \widehat u(\ell) \|_{H^s_x} \Big)^{\frac12} = \Big(\sum_{\ell \in \Z^\nu} \sum_{j \in \Z^d} \langle \ell \rangle^{2 \sigma} \langle j \rangle^{2 s} |\widehat u(\ell, j)|^2 \Big)^{\frac12}\,.
\end{equation}
Similarly to \eqref{equivalenza sobolev 0}, one has that for $u = (u_1, \ldots, u_n) \in H^\sigma\Big(\T^\nu, H^s(\T^d, \R^n) \Big)$
\begin{equation}\label{equivalenza sobolev 1}
\| u \|_{\sigma, s} \simeq {\rm max}_{i = 1, \ldots, n} \| u_i \|_{\sigma, s}\,. 
\end{equation}
If $\sigma >\nu /2$, then 
\begin{equation}\label{sobolev embedding 1}
\begin{aligned}
& H^\sigma\Big( \T^\nu, H^s(\T^d, \R^n) \Big) \quad \text{is compactly embedded in} \quad {\cal C}^0 \Big(\T^\nu,  H^s(\T^d, \R^n) \Big)\,, \\
& \text{and} \quad \| u \|_{{\cal C}^0_\vphi H^s_x} \lesssim_\sigma \| u \|_{H^\sigma_\vphi H^s_x}\,. 
\end{aligned}
\end{equation}
Moreover, the following standard algebra property holds. 
\begin{lemma}\label{lemma algebra}
Let $\sigma > \frac{\nu}{2}$, $s > \frac{d}{2}$, $u, v \in H^\sigma\Big(\T^\nu, H^s(\T^d, \R^n) \Big)$. Then  $u \cdot v \in H^\sigma\Big(\T^\nu, H^s(\T^d, \R) \Big)$ and $\| u \cdot  v \|_{\sigma, s} \lesssim_{\sigma, s} \| u \|_{\sigma, s} \| v \|_{\sigma, s}$. 
\end{lemma}
For any $u \in L^2(\T^d, \R^n)$ we define the orthogonal projections $\pi_0$ and $\pi_0^\bot$ as
\begin{equation}\label{propiettori media nulla}
\begin{aligned}
\pi_0 u := \frac{1}{(2 \pi)^d} \int_{\T^d} u(x)\, d x = \widehat u(0) \quad \text{and} \quad \pi_0^\bot u  := u - \pi_0 u= \sum_{\xi \in \Z^d \setminus \{ 0 \}} \widehat u(\xi) e^{\ii x \cdot \xi}\,. 
\end{aligned}
\end{equation}
According to \eqref{propiettori media nulla}, \eqref{fourier tempo 3}, every function $u \in L^2\Big(\T^\nu, L^2(\T^d, \R^n) \Big)$ can be decomposed as 
\begin{equation}\label{decomposizione media u vphi x}
\begin{aligned}
u(\vphi, x) & = u_0(\vphi) + u_\bot (\vphi, x)\,, \\
u_0(\vphi) & := \pi_0 u(\vphi) = \sum_{\ell \in \Z^\nu} \widehat u(\ell, 0) e^{\ii \ell \cdot \vphi}\,, \\
u_\bot(\vphi, x) & := \pi_0^\bot u(\vphi, x) = \sum_{\ell \in \Z^\nu} \sum_{j \in \Z^d \setminus \{ 0 \}} \widehat u(\ell, j) e^{\ii \ell \cdot \vphi} e^{\ii j \cdot x}\,.
\end{aligned}
\end{equation}
Clearly if $u \in H^\sigma\Big( \T^\nu, H^s(\T^d, \R^n) \Big)$, $\sigma , s \geq 0$, then 
\begin{equation}\label{pi 0 pi 0 bot u norma}
\begin{aligned}
& u_0 \in H^\sigma(\T^\nu, \R^d) \quad \text{and} \quad  \|u_0 \|_\sigma \leq \| u \|_{\sigma, 0} \leq \| u \|_{\sigma, s}\,, \\
&  u_\bot \in H^\sigma\Big( \T^\nu, H^s_0(\T^d, \R^n) \Big) \quad \text{and} \quad \|  u_\bot \|_{\sigma, s} \leq \| u \|_{\sigma, s}\,, \\
& \| u \|_{\sigma, s} = \| u_0 \|_\sigma + \| u_\bot \|_{\sigma, s}\,. 
\end{aligned}
\end{equation}

\noindent
We also prove the following lemma that we shall apply in Section \ref{sezione stabilita asintotica}. 
\begin{lemma}\label{piccolo lemma quasi periodiche}
Let $\sigma > \nu/2$, $U \in H^\sigma \Big(\T^\nu, H^s(\T^d, \R^n) \Big)$ and $\omega \in \R^\nu$. Defining $u_\omega (t, x) := U(\omega t, x)$, $(t, x) \in \R \times \T^d$, one has that $u_\omega \in {\cal C}^0_b\Big(\R, H^s(\T^d, \R^n) \Big)$ and $\| u_\omega \|_{{\cal C}^0_t H^s_x} \lesssim_\sigma \| U \|_{\sigma, s}$.
\end{lemma}
\begin{proof}
By the Sobolev embedding \eqref{sobolev embedding 1}, and using that the map $\R \to \T^d$, $t \mapsto \omega t$ is continuous, one has that $u_\omega \in {\cal C}^0_b\Big(\R, H^s(\T^d, \R^n) \Big)$ and 
$$
\| u_\omega \|_{{\cal C}^0_t H^s_x} \leq \| U \|_{{\cal C}^0_\vphi H^s_x} \lesssim_\sigma \| U \|_{H^\sigma_\vphi H^s_x}\,. 
$$
\end{proof}

\subsection{Leray projector and some elementary properties of the Navier-Stokes equation}
We introduce the space of zero-divergence vector fields
\begin{equation}\label{zero divergenza spazio}
{\cal D}_0(\T^d) := \Big\{ u \in L^2(\T^d, \R^d) : {\rm div}(u) = 0 \Big\}
\end{equation} 
where clearly the divergence has to be interpreted in a distributional sense. The $L^2$-orthogonal projector on this subspace of $L^2(\T^d, \R^d)$ is called the {\it Leray} projector and its explicit formula is given by 
\begin{equation}\label{def proiettore leray}
\begin{aligned}
& {\frak L } : L^2(\T^d, \R^d) \to {\cal D}_0(\T^d)\,, \\
& {\frak L}(u) := u + \nabla (- \Delta)^{- 1} {\rm div}(u)
\end{aligned}
\end{equation}
where the inverse of the laplacian (on the space of zero average functions) $(- \Delta)^{- 1}$ is defined by 
\begin{equation}\label{inverso laplaciano}
(- \Delta)^{- 1} u(x) := \sum_{\xi \in \Z^d \setminus \{ 0 \}} \frac{1}{|\xi|^2} \widehat u(\xi) e^{\ii x \cdot \xi}\,. 
\end{equation}
By expanding in Fourier series, the Leray projector ${\frak L}$ can be written as 
\begin{equation}\label{Leray Fourier}
{\frak L}(u)(x) = u(x) + \sum_{\xi \in \Z^d \setminus \{ 0 \}} \frac{\xi}{|\xi|^2} \xi \cdot \widehat u(\xi) e^{\ii x \cdot \xi}\,. 
\end{equation}
By the latter formula, one immediately deduces some elementary properties of the Leray projector ${\frak L}$. One has  
\begin{equation}\label{media proiettore di Leray}
\int_{\T^d} {\frak L}(u)(x)\, d x = \int_{\T^d} u(x)\, d x, \quad \forall u \in L^2(\T^d, \R^d)
\end{equation}
and for any Fourier multipier $\Lambda$, $\Lambda u(x) = \sum_{\xi \in \Z^d} \Lambda (\xi) \widehat u(\xi) e^{\ii x \cdot \xi}$, the commutator 
\begin{equation}\label{commutatore Leray fourier multiplier}
[{\frak L}, \Lambda] = {\frak L} \Lambda - \Lambda {\frak L} = 0\,. 
\end{equation}
Moreover 
\begin{equation}\label{boundedness Leray projector}
\begin{aligned}
& \|{\frak L} (u) \|_{H^s_x} \lesssim \| u \|_{H^s_x}, \quad \forall u \in H^s(\T^d, \R^d)\,, \\
& \| {\frak L}(u) \|_{\sigma, s} \lesssim \| u \|_{\sigma, s}, \quad \forall u \in H^\sigma\Big(\T^\nu, H^s(\T^d, \R^d) \Big) \,. 
\end{aligned}
\end{equation}
For later purposes, we now prove the following Lemma. 
\begin{lemma}\label{prop nonlinearita per quasi periodiche}

\noindent
$(i)$ Let $u, v \in H^1(\T^d, \R^d)$ and assume that ${\rm div}(u) = 0$, then $u \cdot \nabla v$, ${\frak L}(u \cdot \nabla v)$ have zero average.

\noindent
$(ii)$ Let $\sigma > \nu/2$, $s > d/2$, $u \in H^\sigma\Big( \T^\nu, H^s(\T^d, \R^d) \Big)$, $v \in H^\sigma\Big( \T^\nu, H^{s + 1}(\T^d, \R^d) \Big)$. Then $u \cdot \nabla v \in  H^\sigma\Big( \T^\nu, H^s(\T^d, \R^d) \Big)$ and $\| u \cdot \nabla v\|_{\sigma, s} \lesssim_{\sigma, s} \| u \|_{\sigma, s} \| v \|_{\sigma, s + 1}$. 
\end{lemma}
\begin{proof}
{\sc Proof of $(i)$.} By integrating by parts, 
$$
\begin{aligned}
\int_{\T^d}{\frak L}( u \cdot \nabla v)\, d x \stackrel{\eqref{media proiettore di Leray}}{=} \int_{\T^d} u \cdot \nabla v\, d x = - \int_{\T^d} {\rm div}(u) v\, d x  = 0\,. 
\end{aligned}
$$
{\sc Proof of $(ii)$.} For $u = (u_1, \ldots, u_d)$, $v = (v_1, \ldots, v_d)$, the vector field $u \cdot \nabla v$ is given by 
$$
u \cdot \nabla v = \Big(u \cdot \nabla v_1, u \cdot \nabla v_2, \ldots, u \cdot \nabla v_d \Big)\,.
$$
Then the claimed statement follows by \eqref{equivalenza sobolev 1} and the algebra Lemma \ref{lemma algebra}. 
\end{proof}

\section{Construction of quasi-periodic solutions}\label{sezione costruzione quasi-periodiche}
We look for quasi periodic solutions $u_\omega(t, x)$, $p_\omega(t,x)$ of the equation \eqref{Eulero1}, oscillating with frequency $\omega = (\omega_1, \ldots, \omega_\nu) \in \R^\nu$, namely we look for $u_\omega(t, x) := U(\omega t, x)$, $p_\omega(t, x) := P(\omega t, x)$ where $U : \T^\nu \times \T^d \to \R^d$ and $P : \T^\nu \times \T^d \to \R$ are smooth functions. This leads to solve a functional equation for $U(\vphi, x)$, $P(\vphi, x)$ of the form 
\begin{equation}\label{Euleroqp}
\begin{cases}
\omega \cdot \partial_\vphi  U - \Delta U + U \cdot \nabla U  + \nabla P = \e f(\vphi, x) \\
\div U = 0\,.
\end{cases} 
\end{equation}
If we take the divergence of the latter equation, one gets 
\begin{equation}\label{equazione per la pressione}
\Delta P = {\rm div}\Big(\e f - U \cdot \nabla U \Big)
\end{equation}
and by projecting on the space of zero divergence vector fields, one gets a closed equation for $U$ of the form 
\begin{equation}\label{equazione chiusa per u a}
\omega \cdot \partial_\vphi  U - \Delta U + {\frak L}(U \cdot \nabla U) =   \e {\frak L}(f), \quad U(\vphi, \cdot) \in {\cal D}_0(\T^d) 
\end{equation}
where we recall the definitions \eqref{zero divergenza spazio}, \eqref{def proiettore leray}. According to the splitting \eqref{decomposizione media u vphi x} and by applying the projectors $\pi_0, \pi_0^\bot$ to the equation \eqref{equazione chiusa per u a} one gets the decoupled equations 
\begin{equation}\label{equazione mediata}
\omega \cdot \partial_\vphi U_0(\vphi) = \e f_0(\vphi)
\end{equation}
and 
\begin{equation}\label{equazione chiusa per u}
\omega \cdot \partial_\vphi  U_\bot - \Delta U_\bot + {\frak L}(U_\bot  \cdot \nabla U_\bot) =   \e {\frak L}(f_\bot)\,.
\end{equation}
Then, since $\omega $ is diophantine (see \eqref{def diofanteo}) and using that 
$$
\int_{\T^\nu} f_0(\vphi)\, d \vphi = \int_{\T^\nu \times \T^d} f(\vphi, x)\, d \vphi\, d x \stackrel{\eqref{condizioni media f}}{=} 0, 
$$
($\widehat f(0, 0) = 0$) the averaged equation \eqref{equazione mediata} can be solved explicitely by setting  
\begin{equation}\label{formula U0}
U_0(\vphi) := (\omega \cdot \partial_\vphi)^{- 1} f_0(\vphi) = \sum_{\ell \in \Z^\nu \setminus \{ 0 \}} \frac{\widehat f(\ell, 0)}{\ii \omega \cdot \ell} e^{\ii \ell \cdot \vphi}\,.
\end{equation}
By \eqref{pi 0 pi 0 bot u norma} and using \eqref{def diofanteo}, one gets the estimate 
\begin{equation}\label{stima Sobolev U0}
\| U_0 \|_\sigma \leq \e \gamma^{- 1}  \| f_0 \|_{\sigma + \nu}\, \leq \e \gamma^{- 1} \| f \|_{\sigma + \nu, 0}. 
\end{equation}
\begin{remark}[\bf Non resonance conditions]
The diophantine condition \eqref{def diofanteo} on the frequency vector $\omega$ is used only to solve the averaged equation \eqref{equazione mediata}. In order to solve the equation \eqref{equazione chiusa per u} on the space of zero average functions (w.r. to $x$) no resonance conditions are required. 
\end{remark}
We now solve the equation \eqref{equazione chiusa per u} by means of a fixed point argument. To this aim, we need to analyze some invertibility properties of the linear operator 
\begin{equation}\label{L omega}
L_\omega := \omega \cdot \partial_\vphi - \Delta\,. 
\end{equation} 
\begin{lemma}[\bf Invertibility of $L_\omega$]\label{invertibilita L omega}
Let $\sigma, s \geq 0$, $g \in H^\sigma\Big(\T^\nu, H^s_0(\T^d, \R^d) \Big)$ and assume that $g$ has zero divergence. Then there exists a unique $u  := L_\omega^{- 1} g \in H^\sigma\Big(\T^\nu, H^{s + 2}_0(\T^d, \R^d) \Big)$ with zero divergence which solves the equation $L_\omega u = g$. Moreover 
\begin{equation}\label{stima L omega inverse}
\| u \|_{\sigma, s + 2} \leq \| g \|_{\sigma, s}\,.
\end{equation}
\end{lemma}
\begin{proof}
By \eqref{fourier tempo 3}, we can write
$$
L_\omega u(\vphi, x) = \sum_{\ell \in \Z^\nu} \sum_{j \in \Z^3 \setminus \{ 0 \}} \big(\ii \omega \cdot \ell + |j|^2 \big)\widehat u(\ell, j) e^{\ii \ell \cdot \vphi} e^{\ii j \cdot x}\,.
$$
Note that since $j \neq 0$, one has that 
\begin{equation}\label{lower bounds divisori}
|\ii \omega \cdot \ell + |j|^2| = \sqrt{|\omega \cdot \ell|^2 + |j|^4} \geq |j|^2\,. 
\end{equation}
Hence, the equation $L_\omega u = g$ admits the unique solution with zero space average given by 
\begin{equation}\label{def L omega inv nella dim}
u (\vphi, x) := L_\omega^{- 1} g(\vphi, x) =  \sum_{\ell \in \Z^\nu} \sum_{j \in \Z^d \setminus \{ 0 \}} \dfrac{\widehat g(\ell, j)}{\ii \omega \cdot \ell + |j|^2} e^{\ii \ell \cdot \vphi} e^{\ii j \cdot x}
\end{equation}
Clearly if ${\rm div}(g) = 0$ and then also ${\rm div}(u) = 0$. We now estimate $\| u \|_{\sigma, s + 2}$. According to \eqref{norma sobolev mista}, \eqref{def L omega inv nella dim},  one has 
$$
\begin{aligned}
\| u \|_{\sigma, s + 2}^2 &= \sum_{\ell \in \Z^\nu} \sum_{j \in \Z^d \setminus \{ 0 \}} \langle \ell \rangle^{2 \sigma} \langle j \rangle^{2(s + 2)}\dfrac{|\widehat g(\ell, j)|^2}{|\ii \omega \cdot \ell + |j|^2|^2} \\
& \stackrel{\eqref{lower bounds divisori}}{\leq} \sum_{\ell \in \Z^\nu} \sum_{j \in \Z^d \setminus \{ 0 \}} \langle \ell \rangle^{2 \sigma} | j |^{2(s + 2)} |j|^{- 4}|\widehat g(\ell, j)|^2 = \| g \|_{\sigma, s}^2
\end{aligned}
$$
which proves the claimed statement. 
\end{proof}
We now implement the fixed point argument for the equation \eqref{equazione chiusa per u} (to simplify notations we write $U$ instead of $U_\bot$). For any $\sigma, s, R \geq 0$, we define the ball 
\begin{equation}\label{palla punto fisso}
{\cal B}_{\sigma, s}(R) := \Big\{ U \in H^\sigma\Big(\T^\nu, H^s_0(\T^d, \R^d) \Big) : {\rm div}(U) = 0\,, \quad  \| U \|_{\sigma, s} \leq R\Big\}\,. 
\end{equation}
and we define the nonlinear operator 
\begin{equation}\label{mappa punto fisso}
\Phi(U) := L_\omega^{- 1} {\frak L}\Big(\e f - U \cdot \nabla U \Big), \quad U  \in {\cal B}_{\sigma, s}(R)\,. 
\end{equation}
The following Proposition holds.
\begin{proposition}[\bf Contraction for $\Phi$]\label{contrazione quasi periodica}
Let $\sigma > \nu/2$, $s  > d/2 + 1$, $f \in {\cal C}^N(\T^\nu \times \T^d, \R^d)$, $N > \sigma + s - 2$. Then there exists a constant $C_* = C_*(f , \sigma, s) > 0$ large enough and $\e_0= \e_0(f , \sigma, s) \in (0, 1)$ small enough, such that for any $\e \in (0,  \e_0)$, the map $\Phi : {\cal B}_{\sigma, s}(C_* \e) \to {\cal B}_{\sigma, s}(C_* \e)$ is a contraction. 
\end{proposition}
\begin{proof}
Let $U \in {\cal B}_{\sigma, s} (C_* \e)$. We apply Lemmata \ref{prop nonlinearita per quasi periodiche}-$(i)$, \ref{invertibilita L omega} from which one immediately deduces that 
\begin{equation}\label{parco lambro 0}
\int_{\T^3 } \Phi(U)\, d x = 0, \quad {\rm div}\big(\Phi(U) \big) = 0\,.
\end{equation}
Moreover 
$$
\begin{aligned}
\| \Phi(U) \|_{\sigma, s} & = \Big\| L_\omega^{- 1} {\frak L}\Big(\e f - U \cdot \nabla U \Big) \Big\|_{\sigma, s} \stackrel{\eqref{boundedness Leray projector}, \eqref{stima L omega inverse}}{\lesssim} \Big\| \e f - U \cdot \nabla U \Big\|_{\sigma , s - 2} \\
& \lesssim \e \| f \|_{\sigma, s - 2} + \| U \cdot \nabla U \|_{\sigma, s - 1} \,.
\end{aligned}
$$
Note that since $f \in {\cal C}^N$ with $N > \sigma + s - 2$, one has that $\| f \|_{\sigma, s - 2} \lesssim \| f \|_{{\cal C}^N}$. In view of Lemma \ref{prop nonlinearita per quasi periodiche}-$(ii)$, using that $\sigma > \nu/2$, $s - 1 > d/2$, one gets that 
$$
\| \Phi(U ) \|_{\sigma, s} \leq C(f, s, \sigma) \big(\e +  \| U \|_{\sigma, s - 1}\| U \|_{\sigma, s} \big) \leq C(f, s, \sigma) \big(\e + \| U \|_{\sigma, s}^2 \big)
$$
for some constant $C(f, s, \sigma) > 0$. Using that $\| U \|_{\sigma, s} \leq C_* \e$, one gets that 
$$
\| \Phi(U ) \|_{\sigma, s} \leq C(f, s, \sigma) \e + C(f, s, \sigma) C_*^2 \e^2 \leq C_* \e
$$
provided 
$$
C_* \geq 2 C(f, s, \sigma)\quad \text{and} \quad \e \leq \frac{1}{2 C(f , s, \sigma) C_*}\,. 
$$
Hence $\Phi : {\cal B}_{\sigma, s}(C_* \e) \to {\cal B}_{\sigma, s}(C_* \e)$. Now let $U_1, U_2 \in {\cal B}_{\sigma, s}(C_* \e)$ and we estimate 
$$
\Phi(U_1) - \Phi(U_2) = L_\omega^{- 1} {\frak L} \Big( U_1 \cdot \nabla U_1 - U_2 \cdot \nabla U_2 \Big)\,. 
$$
One has that 
$$
\begin{aligned}
\| \Phi(U_1) - \Phi(U_2) \|_{\sigma, s} & \leq \Big\| L_\omega^{- 1} {\frak L} \Big(( U_1 - U_2)\cdot \nabla U_1 \Big) \Big\|_{\sigma, s} + \Big\|L_\omega^{- 1} {\frak L} \Big(U_2\cdot \nabla (U_1 - U_2) \Big) \Big\|_{\sigma, s}  \\
& \stackrel{\eqref{boundedness Leray projector}, \eqref{stima L omega inverse}\,,Lemma\, \ref{prop nonlinearita per quasi periodiche}}{\lesssim_{s, \sigma}} \| U_1 - U_2 \|_{\sigma, s } \| U_1 \|_{\sigma, s } + \| U_1 - U_2 \|_{\sigma, s } \| U_2 \|_{\sigma, s }  \\
& \leq C(s, \sigma)\big( \| U_1 \|_{\sigma, s} + \| U_2 \|_{\sigma, s} \big) \| U_1 - U_2 \|_{\sigma, s}
\end{aligned}
$$
for some constant $C(s, \sigma) > 0$. Since $U_1, U_2 \in {\cal B}_{\sigma, s}(C_* \e)$, one then has that 
$$
\| \Phi(U_1) - \Phi(U_2) \|_{\sigma, s} \leq 2 C(s, \sigma) C_* \e \| U_1 - U_2 \|_{\sigma, s} \leq \frac12 \| U_1 - U_2 \|_{\sigma, s}
$$
provided $\e \leq \frac{1}{4 C(s, \sigma) C_*}\,.$ Hence $\Phi$ is a contraction. 
\end{proof}
\subsection{Proof of Theorem \ref{esistenza quasi-periodiche}}
Proposition \ref{contrazione quasi periodica} implies that for $\sigma > \nu /2$, $s > \frac{d}{2} + 1$, there exists a unique $U_\bot \in H^\sigma(\T^\nu, H^s_0(\T^d, \R^d))$, $\| U_\bot \|_{\sigma, s} \lesssim_{\sigma, s} \e$ which is a fixed point of the map $\Phi$ defined in \eqref{mappa punto fisso}. We fix $\sigma := \nu/2 + 2$ and $N > \frac{3 \nu}{2} + s + 2 $. By the Sobolev embedding property \eqref{sobolev embedding 1}, since $\sigma - 1 > \nu/2$, one gets that 
\begin{equation}\label{U C1 teo qp}
\begin{aligned}
U_\bot \in {\cal C}^1_b\Big(\T^\nu, H^s_0(\T^d, \R^d) \Big), \quad \| U_\bot \|_{{\cal C}^1_\vphi H^s_x} \lesssim_s \e
\end{aligned}
\end{equation}  and $U_\bot$ is a classical solution of the equation \eqref{equazione chiusa per u}. Similarly, by recalling \eqref{equazione mediata}, \eqref{formula U0}, \eqref{stima Sobolev U0}, one gets that 
\begin{equation}\label{U0 C1 teo qp}
U_0 \in {\cal C}^1 (\T^\nu, \R^d), \quad \| U_0 \|_{{\cal C}^1_\vphi} \leq \e \gamma^{- 1} \| f \|_{\frac{3 \nu}{2} + 2, 0} \leq \e \gamma^{- 1} \| f \|_{{\cal C}^N}, \quad \int_{\T^\nu} U_0(\vphi)\, d \vphi = 0
\end{equation} 
and $U_0$ is a classical solution of the equation \eqref{equazione mediata}. Hence $U = U_0 + U_\bot \in {\cal C}^1\Big( \T^\nu, H^s(\T^d, \R^d) \Big)$ is a classical solution of \eqref{equazione chiusa per u a} and it satisfies $\int_{\T^\nu \times \T^d} U(\vphi, x)\, d \vphi\, d x = 0$. The unique solution with zero average in $x$ of the equation \eqref{equazione per la pressione} is given by
$$
P:=   (- \Delta)^{- 1}{\rm div}\Big( U \cdot \nabla U - \e f \Big)\,. 
$$
Hence, $P \in {\cal C}^0_b\Big(\T^\nu, H^s_0(\T^d, \R^d) \Big)$ and
$$
\begin{aligned}
\| P \|_{{\cal C}^0_\vphi H^s_x} & \stackrel{\eqref{sobolev embedding 1}, \sigma = \nu/2 + 2}{\lesssim}\| P \|_{\sigma, s}  \lesssim_{\sigma, s} \e \| f \|_{\sigma, s - 1} + \| U \cdot \nabla U \|_{\sigma , s - 1}   \\
& \stackrel{Lemma\, \ref{prop nonlinearita per quasi periodiche}}{\lesssim_{\sigma, s}} \e \| f \|_{\sigma, s - 1} + \| U \|_{\sigma, s}^2\,. 
\end{aligned}
$$
The claimed estimate on $P$ then follows since $ \| f \|_{\sigma, s - 1}  \lesssim \| f \|_{{\cal C}^N}$, $\| U \|_{\sigma, s} \leq C_* \e$. Note that if $f$ has zero average in $x$, one has that 
$$
f_0(\vphi) = \pi_0 f(\vphi) = \frac{1}{(2 \pi)^d} \int_{\T^d} f(\vphi, x)\, d x = 0\,, \quad \forall \vphi \in \T^\nu\,.  
$$
The equation \eqref{equazione mediata} reduces to $\omega \cdot \partial_\vphi U_0 = 0$. Hence the only solution $U = U_0 + U_\bot$ of \eqref{equazione chiusa per u a} with zero average in $x$ is the one where we choose $U_0 = 0$ and hence $U = U_\bot$. The claimed statement has then been proved. 
\section{Orbital and asymptotic stability}\label{sezione stabilita asintotica}
We now want to study the Cauchy problem for the equation \eqref{Eulero1} for initial data which are close to the quasi-periodic solution $(u_\omega, p_\omega)$, where 
\begin{equation}\label{u omega t p omega t}
u_\omega (t, x) : = U(\omega t, x), \quad p_\omega (t, x) := P(\omega t, x)
\end{equation}
and the periodic functions $U \in {\cal C}^1\Big(\T^\nu, H^s(\T^d, \R^d) \Big)$, $P\in {\cal C}^0\Big(\T^\nu, H^s_0(\T^d, \R^d) \Big)$ are given by Theorem \ref{esistenza quasi-periodiche}. We then look for solutions which are perturbations of the quasi-periodic ones $(u_\omega, p_\omega)$, namely we look for solutions of the form 
\begin{equation}\label{ansatz perturbazione quasi-periodiche}
u(t, x) = u_\omega (t, x) + v( t, x), \quad p(t, x) = p_\omega( t, x) + q(t, x)\,.
\end{equation}
Plugging the latter ansatz into the equation \eqref{Eulero1}, one obtains an equation for $v(t, x)$, $q(t, x)$ of the form 
\begin{equation}\label{NS per stabilita asintotica}
\begin{cases}
\partial_t v - \Delta v + u_\omega \cdot \nabla v + v \cdot \nabla u_\omega + v \cdot \nabla v + \nabla q = 0 \\
{\rm div}(v) = 0\,. 
\end{cases}
\end{equation}
If we take the divergence in the latter equation we get the equation for the pressure $q(t, x)$
\begin{equation}\label{eq per pressione stabilita}
- \Delta q =  {\rm div}\Big( u_\omega \cdot \nabla v + v \cdot \nabla u_\omega + v \cdot \nabla v\Big)\,. 
\end{equation}
By using the Leray projector defined in \eqref{def proiettore leray}, we then get a closed equation for $v$ of the form 
\begin{equation}\label{eq per u stab asintotica}
\begin{cases}
\partial_t v - \Delta v + {\frak L}\Big(u_\omega \cdot \nabla v + v \cdot \nabla u_\omega + v \cdot \nabla v \Big)= 0 \\
{\rm div}(v) = 0\,.
\end{cases}
\end{equation}
We prove the following 
\begin{proposition}\label{proposizione stabilita asintotica con Leray}
Let $s > d/2 + 1$, $\alpha \in (0, 1)$. Then there exists $\delta = \delta(s, \alpha,  d, \nu)\in (0, 1)$ small enough and $C = C(s, \alpha, d, \nu) > 0$ large enough, such that for any $\e \in (0, \delta)$ and for any initial datum $v_0 \in H^s_0(\T^d, \R^d)$ with $\| v_0 \|_{H^s_x} \leq \delta$, there exists a unique global classical solution 
\begin{equation}\label{u regolarita}
v \in {\cal C}^0_b \Big( [0, + \infty), H^s_0(\T^d, \R^d) \Big) \cap {\cal C}^1_b\Big( [0, + \infty), H^{s - 2}_0(\T^d, \R^d) \Big)
\end{equation}
of the equation \eqref{eq per u stab asintotica} which satisfies 
\begin{equation}\label{stima asintotica 1}
\| v (t, \cdot) \|_{H^s_x}, \| \partial_t v (t, \cdot )\|_{H^{s - 2}_x}\leq C\delta e^{- \alpha t}, \quad \forall t \geq 0\,. 
\end{equation}
\end{proposition}
The Proposition above will be proved by a fixed point argument in some weighted Sobolev spaces which take care of the decay in time of the solutions we are looking for. In the next section we shall exploit some decay estimates of the linear heat propagator which will be used in the proof of our result. 
\subsection{Dispersive estimates for the heat propagator}\label{stime dispersive propagatore calore}
In this section we analyze some properties of the heat propagator. We recall that the heat propagator is defined as follows. Consider the Cauchy problem for the heat equation 
\begin{equation}\label{calore}
\begin{cases}
\partial_t u - \Delta u = 0 \\
u(0, x) = u_0(x),
\end{cases}
\quad u_0 \in H^s_0(\T^d, \R^d)\,.
\end{equation}
It is well known that there exists a unique solution 
$$
u \in {\cal C}_b^0\Big([0, + \infty), H^s_0(\T^d, \R^d) \Big) \cap {\cal C}_b^1\Big([0, + \infty), H^{s - 2}_0(\T^d, \R^d) \Big)
$$
 which can be written as $u(t, x) := e^{t \Delta} u_0(x)$, namely  
\begin{equation}\label{heat propagator}
u(t, x) = e^{t \Delta} u_0(x) = \sum_{\xi \in \Z^d \setminus \{ 0 \}} e^{- t |\xi|^2} \widehat u_0(\xi) e^{\ii x \cdot \xi}\,. 
\end{equation}
\begin{lemma}\label{lemma propagatore libero calore}

\noindent
$(i)$  Let $u_0 \in H^s_0(\T^d, \R^d)$. Then
\begin{equation}\label{stima banale propagatore calore}
\| e^{t \Delta} u_0 \|_{H^s_x} \leq e^{- t} \| u_0 \|_{H^s_x}, \quad \forall t \geq 0\,.
\end{equation}

\noindent
$(ii)$ Let $u_0 \in H^{s - 1}_0(\T^d, \R^d)$. Then, for any integer $n \geq 1$ and for any $\alpha \in (0, 1)$, 
\begin{equation}\label{stima smoothing propagatore calore}
\| e^{t \Delta} u_0 \|_{H^{s }_x} \lesssim_n t^{- \frac{n}{2}} (1 - \alpha)^{- \frac{n}{2}} e^{- \alpha t}\| u_0 \|_{H^{s - n}_x} \lesssim_n t^{- \frac{n}{2}} (1 - \alpha)^{- \frac{n}{2}} e^{- \alpha t} \| u_0 \|_{H^{s - 1}_x}, \quad \forall t > 0\,. 
\end{equation}
\end{lemma}
\begin{proof}
The item $(i)$ follows by \eqref{heat propagator}, using that $e^{- t |\xi|^2} \leq e^{- t}$ for any $t \geq 0$, $\xi \in \Z^d \setminus \{ 0 \}$, since $|\xi|^2 \geq 1$. We now prove the item $(ii)$. Let $n \in \N$, $\alpha \in (0, 1)$. One has 
\begin{equation}\label{marcantonio 0}
\begin{aligned}
\| e^{t \Delta} u_0 \|_{H^{s}_x}^2 & = \sum_{\xi \in \Z^d \setminus \{ 0 \}} | \xi |^{2 s } e^{- 2 t |\xi|^2} |\widehat u_0(\xi)|^2 \\
& =  \sum_{\xi \in \Z^d \setminus \{ 0 \}} e^{-  2 \alpha t |\xi|^2} | \xi |^{2 (s - n)}  |\xi|^{2n} e^{- 2(1 - \alpha) t |\xi|^2} |\widehat u_0(\xi)|^2\,. 
\end{aligned}
\end{equation}
Using that for any $\xi \in \Z^d \setminus \{ 0 \}$, $t \geq  0$, $e^{- 2 \alpha t |\xi|^2} \leq e^{- 2 \alpha t}$, by \eqref{marcantonio 0}, one gets that
\begin{equation}\label{marcantonio - 1}
\| e^{t \Delta} u_0 \|_{H^{s}_x}^2 \leq  e^{- 2 \alpha t} \sum_{\xi \in \Z^d \setminus \{ 0 \}} | \xi |^{2 (s - n)} |\xi|^{2n} e^{- 2(1 - \alpha) t |\xi|^2} |\widehat u_0(\xi)|^2
\end{equation}
By Lemma \ref{appendix0} (applied with $\zeta = 2 (1 - \alpha) t$), one has that 
$$
\sup_{\xi \in \Z^d \setminus \{ 0 \}} |\xi|^{2n} e^{- 2(1 - \alpha) t |\xi|^2} \leq \sup_{y \geq 0} y^n e^{- 2(1 - \alpha) t  y} \leq \frac{C(n)}{(1 - \alpha)^n t^n}
$$
for some constant $C(n)> 0$. Therefore by \eqref{marcantonio - 1}, one gets that 
$$
\begin{aligned}
\| e^{t \Delta} u_0 \|_{H^{s }_x}^2 & \lesssim_n  t^{- n} e^{- 2 \alpha t} (1 - \alpha)^{- n}\sum_{\xi \in \Z^d \setminus \{ 0 \}} | \xi |^{2 (s - n)}  |\widehat u_0(\xi)|^2  \\
& \lesssim_n  t^{- n} e^{- 2 \alpha t} (1 - \alpha)^{- n}\| u_0 \|_{H^{s - n}_x}^2\,. 
\end{aligned}
$$
The second inequality in \eqref{stima smoothing propagatore calore} clearly follows since $\| \cdot \|_{H^{s - n}_x} \leq \| \cdot \|_{H^{s - 1}_x}$ for $n \geq 1$. 
\end{proof}
We fix $\alpha \in (0, 1)$ and for any $s \geq 0$, we define the space 
\begin{equation}\label{def cal Es}
 {\cal E}_{s} := \Big\{ u \in {\cal C}_b^0\Big([0, + \infty), H^s_0(\T^d, \R^d) \Big) : \| u \|_{{\cal E}_s} := \sup_{t \geq 0} e^{\alpha t} \| u(t) \|_{H^s_x} \Big\}\,.
\end{equation}
Clearly 
\begin{equation}\label{monotonia norme}
\begin{aligned}
& \| \cdot \|_{{\cal E}_s} \leq \| \cdot \|_{{\cal E}_{s'}}, \quad \forall s \leq s'\,. 
\end{aligned}
\end{equation}
The following elementary lemma holds: 
\begin{lemma}\label{proprieta elementari cal Es}
$(i)$ Let $u \in {\cal E}_s$. Then 
\begin{equation}\label{inclusione cal Es}
\begin{aligned}
&  \| u \|_{{\cal C}^0_t H^s_x} \lesssim \| u \|_{{\cal E}_s} \quad \text{and} \quad \| {\frak L}(u) \|_{{\cal E}_s} \lesssim \| u \|_{{\cal E}_s}\,, \\
& \| u (t) \|_{H^s_x} \leq e^{- \alpha t} \| u \|_{{\cal E}_s}, \quad \forall t \geq 0\,. 
 \end{aligned}
\end{equation}
$(ii)$ Let $s > d/2$, $u \in {\cal E}_s$, $v \in {\cal C}^0_b\Big([0, + \infty), H^{s + 1}(\T^d, \R^d) \Big)$, ${\rm div}(u) = 0$. Then the product $u \cdot \nabla v \in {\cal E}_s$ and
\begin{equation}\label{interpolazione cal Es Hs}
\| u \cdot \nabla  v \|_{{\cal E}_s} \lesssim_s \| u \|_{{\cal E}_s} \| v \|_{{\cal C}^0_t H^{s + 1}_x}\,. 
\end{equation} 
$(iii)$ Let $s > d/2$, $u \in {\cal C}^0_b\Big([0, + \infty), H^s(\T^d, \R^d) \Big)$, ${\rm div}(u) = 0$ and $v \in {\cal E}_{s + 1}$. Then the product $u \cdot \nabla v \in {\cal E}_s$ and
\begin{equation}\label{interpolazione cal Es Hs 1}
\| u \cdot \nabla  v \|_{{\cal E}_s} \lesssim_s \| u \|_{{\cal C}^0_t H^s_x} \| v \|_{{\cal E}_{s + 1}}\,. 
\end{equation}
$(iv)$ Let $s> d/2$, $u \in {\cal E}_s$, ${\rm div}(u)= 0$, $v \in {\cal E}_{s + 1}$. Then $u \cdot \nabla v \in {\cal E}_{s}$ and 
\begin{equation}\label{interpolazione cal Es Hs 2}
\| u \cdot \nabla  v \|_{{\cal E}_s} \lesssim_s \| u \|_{{\cal E}_s} \| v \|_{{\cal E}_{s + 1}}\,. 
\end{equation}
\end{lemma}
\begin{proof}
The item $(i)$ is very elementary and it follows in a straightforward way by the definition \eqref{def cal Es} and by recalling the estimate \eqref{boundedness Leray projector} on the Leray projector ${\frak L}$. We prove the item $(ii)$. By Lemma \ref{prop nonlinearita per quasi periodiche}-$(i)$, one has that $u(t) \cdot \nabla v(t)$ has zero average in $x$. Moreover 
$$
u \cdot \nabla v = \Big( u \cdot \nabla v_1, \ldots, u \cdot \nabla v_d \Big)
$$
therefore, since $s > d/2$, by Lemma \ref{interpolazione sobolev 0} and using \eqref{equivalenza sobolev 0} one has that for any $t \in [0, + \infty)$ $\| u (t) \cdot \nabla v(t) \|_{H^s_x} \lesssim_s \| u(t) \|_{H^s_x} \| v(t) \|_{H^{s + 1}_x}$ implying that 
\begin{equation}\label{cleopatra 0}
\begin{aligned}
e^{\alpha t} \| u(t) \cdot \nabla v(t) \|_{H^s_x} & \lesssim_s e^{\alpha t} \| u(t) \|_{H^s_x} \| v(t) \|_{H^{s + 1}_x} \lesssim_s \Big( \sup_{t \geq 0} e^{\alpha t} \| u(t) \|_{H^s_x} \Big) \Big( \sup_{t \geq 0} \| v(t) \|_{H^{s + 1}_x}\Big) \\
& \lesssim_s \| u \|_{{\cal E}_s} \| v \|_{{\cal C}^0_t H^{s + 1}_x}\,. 
\end{aligned}
\end{equation}
Passing to the supremum over $t \geq 0$ in the left hand side of \eqref{cleopatra 0}, we get the claimed statement. The item $(iii)$ follows by similar arguments and the item $(iv)$ follows by applying items $(i)$ and $(ii)$. 
\end{proof}
We now prove some estimates for the heat propagator $e^{t \Delta}$ in the space ${\cal E}_s$. 
\begin{lemma}\label{stima calore in E sn}
Let $s \geq 0$, $u_0 \in H^s_0(\T^d, \R^d)$. Then $\| e^{t \Delta } u_0 \|_{{\cal E}_s} \lesssim_\alpha \| u_0 \|_{H^s_x}$. 
\end{lemma}
\begin{proof}
We have to estimate uniformly w.r. to $t \geq 0$, the quantity $e^{\alpha t} \| e^{t \Delta} u_0 \|_{H^s_x}$. For $t \in [0, 1]$, since $e^{\alpha t} \leq e^{\alpha} \stackrel{\alpha < 1}{\leq} e$ and by applying Lemma \ref{lemma propagatore libero calore}-$(i)$, one gets that 
\begin{equation}\label{gricia 100}
e^{\alpha t} \| e^{t \Delta } u_0 \|_{H^s_x} \lesssim  \| e^{t \Delta } u_0 \|_{H^s_x} \lesssim \| u_0 \|_{H^s_x}\,, \quad \forall t \in [0, 1]\,. 
\end{equation}
For $t > 1$, by applying Lemma \ref{lemma propagatore libero calore}-$(ii)$ (for $n = 1$), one gets that 
\begin{equation}\label{gricia 101}
e^{\alpha t}\| e^{t \Delta } u_0 \|_{H^s_x}  \lesssim_\alpha t^{- \frac12} \| u_0 \|_{H^{s - 1}_x} \lesssim_\alpha \| u_0 \|_{H^s_x}\,, \quad \forall t > 1\,. 
\end{equation}
Hence the claimed statement follows by \eqref{gricia 100}, \eqref{gricia 101} passing to the supremum over $t \geq 0$. 
\end{proof}
The main result of this section is the following Proposition
\begin{proposition}\label{stima dispersiva principale}
Let $s\geq 1$, $f \in {\cal E}_{s - 1}$ and define 
\begin{equation}\label{integrale duhamel}
u(t) \equiv u(t, \cdot) := \int_0^t e^{(t - \tau)\Delta} f(\tau, \cdot)\, d \tau\,. 
\end{equation}
Then $u \in {\cal E}_s$ and 
\begin{equation}\label{stima principale per punto fisso}
\| u \|_{{\cal E}_s} \lesssim_\alpha \| f \|_{{\cal E}_{s - 1}}\,. 
\end{equation}
\end{proposition}
The proof is split in several steps. The first step is to estimate the integral in \eqref{integrale duhamel} for any $t \in [0, 1]$. 
\begin{lemma}\label{integrale propagatore 1}
Let $t \in [0, 1]$, $f \in {\cal E}_{s - 1}$ and $u$ defined by \eqref{integrale duhamel}. Then 
$$
\| u (t) \|_{H^s_x} \lesssim_\alpha \| f \|_{{\cal C}^0_t H^{s - 1}_x}\,.
$$
\end{lemma}
\begin{proof}
Let $t \in [0, 1]$. Then 
\begin{equation}\label{cesare 0}
\begin{aligned}
\| u (t) \|_{H^s_x} & \leq \int_0^t \Big\| e^{(t - \tau)\Delta} f(\tau, \cdot) \Big\|_{H^s_x} d \tau \stackrel{\eqref{stima smoothing propagatore calore}}{\lesssim_\alpha} \int_0^t e^{- \frac{t - \tau}{2}}\frac{\| f(\tau, \cdot) \|_{H^{s - 1}_x}}{\sqrt{t - \tau}}\, d \tau \\
& \stackrel{e^{- \frac{t - \tau}{2}} \leq 1}{\lesssim_\alpha} \int_0^t \frac{1}{\sqrt{t - \tau}}\, d \tau \| f \|_{{\cal C}^0_t H^{s - 1}_x}\,. 
\end{aligned}
\end{equation}
By making the change of variables $z = t - \tau $, one gets that 
$$
\int_0^t \frac{1}{\sqrt{t - \tau}}\, d \tau = \int_0^t \frac{1}{\sqrt{z}}\,d z \stackrel{t \leq 1}{\leq} \int_0^1 \frac{1}{\sqrt{z}}\, d z = 2
$$
and hence in view of \eqref{cesare 0}, one gets $\| u (t) \|_{H^s_x} \lesssim_\alpha \| f \|_{{\cal C}^0_t H^{s - 1}_x}$ for any $t \in [0, 1]$, which is the claimed stetement. 
\end{proof}
For $t > 1$, we split the integral term in \eqref{integrale duhamel} as 
$$
\int_0^t e^{(t - \tau)\Delta} f(\tau, \cdot)\, d \tau = \int_0^{t - \frac12} e^{(t - \tau)\Delta} f(\tau, \cdot)\, d \tau + \int_{t - \frac12}^t e^{(t - \tau)\Delta} f(\tau, \cdot)\, d \tau
$$
and we estimate separately the two terms in the latter formula. More precisely the first term is estimated in Lemma \ref{integrale propagatore 2} and the second one in Lemma \ref{integrale propagatore 3}. 
\begin{lemma}\label{integrale propagatore 2}
Let $t > 1$. Then 
$$
\Big\| \int_0^{t - \frac12} e^{(t- \tau) \Delta} f(\tau, \cdot)\, d \tau \Big\|_{H^s_x} \lesssim_\alpha e^{- \alpha t} \| f \|_{{\cal E}_{s - 1}}
$$
\end{lemma}
\begin{proof}
Let $t > 1$. One then has 
\begin{equation}\label{cesare 1}
\begin{aligned}
\Big\|  \int_0^{t - \frac12} e^{(t- \tau) \Delta} f(\tau, \cdot)\, d \tau  \Big\|_{H^s_x} & \leq  \int_0^{t - \frac12} \Big\| e^{(t- \tau) \Delta} f(\tau, \cdot) \Big\|_{H^s_x}\, d \tau   \\
& \stackrel{\eqref{stima smoothing propagatore calore}}{\lesssim_{n, \alpha}} \int_0^{t - \frac12} e^{- \alpha (t - \tau)} \frac{1}{(t - \tau)^{\frac{n}{2}}} \| f(\tau, \cdot) \|_{H^{s - 1}_x}\, d \tau \\
& \lesssim_{n, \alpha} e^{- \alpha t} \int_0^{t - \frac12} \frac{1}{(t - \tau)^{\frac{n}{2}}} e^{\alpha \tau} \| f(\tau, \cdot) \|_{H^{s - 1}_x}\, d \tau \\
& \lesssim_{n, \alpha} e^{- \alpha t} \Big(    \int_0^{t - \frac12} \frac{d \tau}{(t - \tau)^{\frac{n}{2}}} \Big)\,\,\sup_{\tau \geq 0} e^{\alpha \tau} \| f(\tau, \cdot) \|_{H^{s - 1}_x}\,.
\end{aligned}
\end{equation}
By choosing $n = 4$ and by making the change of variables $z = t - \tau$, on gets that
$$
\int_0^{t - \frac12} \frac{d \tau}{(t - \tau)^{\frac{n}{2}}} = \int_0^{t - \frac12} \frac{d \tau}{(t - \tau)^{2}} = \int_{\frac12}^{t} \frac{d z}{z^2} \leq \int_{\frac12}^{+ \infty} \frac{d z}{z^2} < \infty\,. 
$$
The latter estimate, together with the estimate \eqref{cesare 1} imply the claimed statement. 
\end{proof}

\begin{lemma}\label{integrale propagatore 3}
Let $t > 1$. Then 
$$
\Big\| \int_{t - \frac12}^{t} e^{(t- \tau) \Delta} f(\tau, \cdot)\, d \tau \Big\|_{H^s_x} \lesssim_\alpha e^{- \alpha t} \| f \|_{{\cal E}_{s - 1}}\,. 
$$
\end{lemma}
\begin{proof}
One has 
\begin{equation}\label{cesare 11}
\begin{aligned}
\Big\| \int_{t - \frac12}^{t} e^{(t- \tau) \Delta} f(\tau, \cdot)\, d \tau \Big\|_{H^s_x} & \leq \int_{t - \frac12}^{t}\Big\|   e^{(t- \tau) \Delta} f(\tau, \cdot) \Big\|_{H^s_x}\, d \tau  \\
& \stackrel{\eqref{stima smoothing propagatore calore}}{\lesssim_\alpha} e^{- \alpha t}\int_{t - \frac12}^t \frac{1}{\sqrt{t - \tau}} e^{\alpha \tau}\| f(\tau, \cdot) \|_{H^{s - 1}_x}\, d \tau \\
& \lesssim_\alpha e^{- \alpha t} \int_{t - \frac12}^t \frac{d \tau}{\sqrt{t - \tau}}  \,\, \Big( \sup_{\tau \geq 0}  e^{\alpha \tau} \| f(\tau, \cdot) \|_{H^{s - 1}_x} \Big)\,. 
\end{aligned}
\end{equation}
By making the change of variables $z = t - \tau$, one gets
$$
 \int_{t - \frac12}^t \frac{d \tau}{\sqrt{t - \tau}}  = \int_0^{\frac12} \frac{d z}{\sqrt{z}} = \sqrt{2}\,. 
$$
The latter estimate, together with \eqref{cesare 11} imply the claimed statement. 
\end{proof}
{\sc Proof of Proposition \ref{stima dispersiva principale}.}
\begin{proof}
For any $t \in [0, 1]$, since $e^{\alpha t} \leq e^{\alpha}\stackrel{\alpha < 1}{\leq} e$, by Lemma \ref{integrale propagatore 1}, one has that 
\begin{equation}\label{cleopatra 1}
\begin{aligned}
e^{\alpha t} \Big\| \int_0^t e^{(t - \tau) \Delta} f(\tau, \cdot)\, d \tau \Big\|_{H^s_x} & \lesssim \Big\| \int_0^t e^{(t - \tau) \Delta} f(\tau, \cdot)\, d \tau \Big\|_{H^s_x} \\
& \lesssim_\alpha \| f \|_{{\cal C}^0_t H^{s - 1}_x} \stackrel{\eqref{inclusione cal Es}}{\lesssim_\alpha} \| f \|_{{\cal E}_{s - 1}}\,. 
\end{aligned}
\end{equation}
For any $t > 1$, by applying Lemmata \ref{integrale propagatore 2}, \ref{integrale propagatore 3}, one gets  
\begin{equation}\label{cleopatra 2}
\begin{aligned}
 e^{\alpha t} \Big\| \int_0^t e^{(t - \tau) \Delta} f(\tau, \cdot)\, d \tau \Big\|_{H^s_x} &  \leq e^{\alpha t} \Big\| \int_0^{t - \frac12} e^{(t - \tau) \Delta} f(\tau, \cdot)\, d \tau \Big\|_{H^s_x} \\
& \quad +  e^{\alpha t} \Big\| \int_{t - \frac12}^t e^{(t - \tau) \Delta} f(\tau, \cdot)\, d \tau \Big\|_{H^s_x} \\
& \lesssim_\alpha  \| f \|_{{\cal E}_{s - 1}}  \,. 
\end{aligned}
\end{equation}
The claimed statement then follows by \eqref{cleopatra 1}, \eqref{cleopatra 2} and passing to the supremum over $t \in [0, + \infty)$. 
\end{proof}

\subsection{Proof of Proposition \ref{proposizione stabilita asintotica con Leray}}
The Proposition \ref{proposizione stabilita asintotica con Leray} is proved by a fixed point argument. For any $\delta > 0$ and $s\geq 0$, we define the ball 
\begin{equation}\label{palla punto fisso stabilita asintotica}
{\cal B}_s(\delta) := \Big\{ v \in {\cal E}_s : {\rm div}(v) = 0, \quad \| v \|_{{\cal E}_s} \leq \delta \Big\}
\end{equation}
and for any $v \in {\cal B}_s(\delta)$, we define the map 
\begin{equation}\label{mappa punto fisso stabilita asintotica}
\begin{aligned}
\Phi(v) & := e^{t \Delta} v_0 + \int_0^t e^{(t - \tau) \Delta} {\cal N}( v)(\tau, \cdot)\, d \tau\,, \\
{\cal N}( v) & := - {\frak L}\Big(u_\omega \cdot \nabla v + v \cdot \nabla u_\omega + v \cdot \nabla v \Big)\,.
\end{aligned}
\end{equation}
\begin{proposition}\label{contrazione punto fisso stabilita asintotica}
Let $s > d/2 + 1$, $\alpha \in (0, 1)$. Then there exists $\delta = \delta(s, \alpha, \nu, d) \in (0, 1)$ small enough such that for any $\e \in (0, \delta)$, $\Phi : {\cal B}_s(\delta) \to {\cal B}_s(\delta)$ is a contraction.
\end{proposition}
\begin{proof}
Since $v_0$ has zero divergence and zero average, then clearly by \eqref{heat propagator}, ${\rm div}\big( e^{t \Delta} v_0\big) = 0$,  $\int_{\T^3} e^{t \Delta} v_0(x)\, d x= 0$\,. Now let $v(t, x)$ be a function with zero average and zero divergence. Clearly ${\rm div}\Big( {\cal N}(v)\Big) = 0$
since in the definition of ${\cal N}( v)$ in \eqref{mappa punto fisso stabilita asintotica}, there is the Leray projector. Moreover using that ${\rm div}(v ) = {\rm div}(u_\omega) = 0$, by Lemma \ref{prop nonlinearita per quasi periodiche}-$(i)$, one gets ${\cal N}(v)$ has zero average and then by \eqref{heat propagator} also $\int_0^t e^{(t - \tau) \Delta} {\cal N}(v)(\tau, \cdot )\, d \tau$ has zero average. 
Hence, we have shown that ${\rm div}(\Phi (v)) = 0$ and $\int_{\T^d} \Phi(v)\, d x= 0$. Let now $\| v \|_{{\cal E}_s} \leq \delta$. We estimate $\| \Phi(v) \|_{{\cal E}_s}$. By recalling \eqref{mappa punto fisso stabilita asintotica}, Lemma \ref{stima calore in E sn}, Proposition \ref{stima dispersiva principale} and Lemma \ref{proprieta elementari cal Es}, one has
\begin{equation}\label{Phi Es 0}
\begin{aligned}
\| \Phi(v) \|_{{\cal E}_s} & \stackrel{\eqref{mappa punto fisso stabilita asintotica}}{\leq} \| e^{t \Delta} v_0 \|_{{\cal E}_s} + \Big\| \int_0^t e^{(t - \tau)\Delta} {\cal N}(v)(\tau, \cdot)\, d \tau \Big\|_{{\cal E}_s} \\
& \lesssim_\alpha \| v_0 \|_{H^s_x} + \| {\cal N}( v) \|_{{\cal E}_{s - 1}}  \\
& \lesssim_{s, \alpha} \| v_0 \|_{H^s_x} + \| u_\omega \|_{{\cal C}^0_t H^s_x} \| v \|_{{\cal E}_s} + \| v \|_{{\cal E}_s}^2\,. 
\end{aligned}
\end{equation}
By the estimates of Theorem \ref{esistenza quasi-periodiche}, by the definition \eqref{u omega t p omega t} and by applying Lemma \ref{piccolo lemma quasi periodiche}, one has 
\begin{equation}\label{prop u omega prop principale}
u_\omega \in {\cal C}^0_b \Big( \R, H^s(\T^d, \R^d) \Big) \quad \text{and} \quad \| u_\omega \|_{{\cal C}^0_t H^s_x} \lesssim_s \e \stackrel{\e \leq \delta}{\lesssim_s} \delta\,.
\end{equation}
Since $v \in {\cal B}_s(\delta)$, the estimate \eqref{Phi Es 0} implies that 
\begin{equation}\label{Phi Es 1}
\| \Phi(v) \|_{{\cal E}_s} \leq C(s, \alpha)\Big( \| v_0 \|_{H^s_x} +  \delta^2\Big)\,.
\end{equation}
Hence $\| \Phi(v) \|_{{\cal E}_s} \leq \delta$ provided 
$$
C(s, \alpha) \| v_0 \|_{H^s_x} \leq \frac{\delta}{2}, \quad C(s, \alpha)\delta \leq \frac{1}{2}\,. 
$$
These conditions are fullfilled by taking $\delta$ small enough and $\| v_0 \|_{H^s_x}\ll \delta$. Hence $\Phi : {\cal B}_s(\delta) \to {\cal B}_s(\delta)$. Now let $v_1, v_2 \in {\cal B}_s(\delta)$. We need to estimate 
\begin{equation}\label{Phi u 1 - u 2 stab asintotica}
\Phi(v_1) - \Phi(v_2) = \int_0^t e^{(t - \tau)\Delta} \Big({\cal N}(v_1) - {\cal N}(v_2)\Big)(\tau, \cdot)\, d \tau\,. 
\end{equation}
By \eqref{mappa punto fisso stabilita asintotica}
\begin{equation}\label{stima cal N u1 - u2}
\begin{aligned}
\| {\cal N}(v_1) - {\cal N}(v_2) \|_{{\cal E}_{s - 1}} & \leq \Big\| {\frak L}\Big( u_\omega \cdot \nabla(v_1 - v_2) \Big) \Big\|_{{\cal E}_{s - 1}} + \Big\| {\frak L}\Big( (v_1 - v_2) \cdot \nabla u_\omega \Big) \Big\|_{{\cal E}_{s - 1}} \\
& \quad  + \Big\| {\frak L}\Big( (v_1 - v_2) \cdot \nabla v_1 \Big) \Big\|_{{\cal E}_{s - 1}} + \Big\| {\frak L}\Big(  v_2 \cdot \nabla (v_1 - v_2) \Big) \Big\|_{{\cal E}_{s - 1}} \\
& \stackrel{\eqref{inclusione cal Es}}{\lesssim} \Big\| u_\omega \cdot \nabla(v_1 - v_2)  \Big\|_{{\cal E}_{s - 1}} + \Big\| (v_1 - v_2) \cdot \nabla u_\omega  \Big\|_{{\cal E}_{s - 1}} \\
& \quad  + \Big\|  (v_1 - v_2) \cdot \nabla v_1 \Big\|_{{\cal E}_{s - 1}} + \Big\|  v_2 \cdot \nabla (v_1 - v_2) \Big\|_{{\cal E}_{s - 1}} \\
& \stackrel{\eqref{interpolazione cal Es Hs}, \eqref{interpolazione cal Es Hs 1}, \eqref{interpolazione cal Es Hs 2}}{\lesssim_s} \Big( \| u_\omega \|_{{\cal C}^0_t H^s_x} + \| v_1 \|_{{\cal E}_s} + \| v_2 \|_{{\cal E}_s} \Big) \| v_1 - v_2 \|_{{\cal E}_s} \,.
\end{aligned}
\end{equation}
Hence, \eqref{prop u omega prop principale} and using that $v_1, v_2 \in {\cal B}_s(\delta)$ ($\| v_1 \|_{{\cal E}_s}, \| v_2 \|_{{\cal E}_s} \leq \delta$) and the estimate \eqref{stima cal N u1 - u2} imply that 
\begin{equation}\label{stima cal N u1 - u2 1}
\| {\cal N}(v_1) - {\cal N}(v_2) \|_{{\cal E}_{s - 1}} \lesssim_{s} \delta \| v_1 - v_2 \|_{{\cal E}_s}\,. 
\end{equation}
By \eqref{Phi u 1 - u 2 stab asintotica}, one gets that 
\begin{equation}\label{Phi u 1 - u 2 stab asintotica 1}
\begin{aligned}
\| \Phi(v_1 ) - \Phi(v_2) \|_{{\cal E}_s} &  \stackrel{Proposition \, \ref{stima dispersiva principale}}{\lesssim_{s, \alpha}} \| {\cal N}(v_1) - {\cal N}(v_2) \|_{{\cal E}_{s - 1}} \\
& \stackrel{\eqref{stima cal N u1 - u2 1}}{\leq} C(s, \alpha)\delta \| v_1 - v_2 \|_{{\cal E}_s}
\end{aligned}
\end{equation}
for some constant $C(s, \alpha) > 0$. Therefore 
$$
\| \Phi(v_1 ) - \Phi(v_2) \|_{{\cal E}_s} \leq \frac12 \| v_1 - v_2 \|_{{\cal E}_s}
$$
provided $\delta \leq \frac{1}{2 C(s, \alpha)}$. The claimed statement has then been proved. 
\end{proof}
{\sc Proof of Proposition \ref{proposizione stabilita asintotica con Leray} concluded.} By Proposition \ref{contrazione punto fisso stabilita asintotica}, using the contraction mapping theorem there exists a unique $v \in {\cal B}_s(\delta)$ which is a fixed point of the map $\Phi$ in \eqref{mappa punto fisso stabilita asintotica}. By the functional equation $v = \Phi(v)$, one deduces in a standard way that 
$$
v \in {\cal C}^1_b \Big( [0, + \infty), H^{s - 2}_0(\T^d, \R^d) \Big)
$$
and hence $v$ is a classical solution of the equation \eqref{eq per u stab asintotica}. By \eqref{inclusione cal Es} and using the trivial fact that $\| \Delta v \|_{{\cal E}_{s - 2}} \leq \| v \|_{{\cal E}_s}$
\begin{equation}\label{toro bla}
\begin{aligned}
\| \partial_t v \|_{{\cal E}_{s - 2}} & \lesssim \| v \|_{{\cal E}_s} + \| u_\omega \cdot \nabla v \|_{{\cal E}_{s - 2}} + \| v \cdot \nabla u_\omega \|_{{\cal E}_{s - 2}} + \| v \cdot \nabla v \|_{{\cal E}_{s - 2}}  \\
& \stackrel{\eqref{interpolazione cal Es Hs}-\eqref{interpolazione cal Es Hs 2}}{\lesssim_s} \| v \|_{{\cal E}_s}\Big(1 + \| u_\omega \|_{{\cal C}^0_t H^s_x} + \| v \|_{{\cal E}_s} \Big)\,. 
\end{aligned}
\end{equation}
Therefore using that $v \in {\cal B}_s(\delta)$ ($\| v \|_{{\cal E}_s} \leq \delta$) and by \eqref{prop u omega prop principale}, $\| u_\omega \|_{{\cal C}^0_t H^s_x} \lesssim_s \delta$, one gets, for $\delta$ small enough, the estimate $\| \partial_t v \|_{{\cal E}_{s - 2}} \lesssim_s \delta$ and the claimed statement follows by recalling \eqref{inclusione cal Es}. 
\subsection{Proof of Theorem \ref{stabilita asintotica}}\label{fine dim stabilita asintotica}
In view of Proposition \ref{proposizione stabilita asintotica con Leray}, it remains only to solve the equation \eqref{eq per pressione stabilita} for the pressure $q(t, x)$. The only solution with zero average of this latter equation is given by 
\begin{equation}\label{sol stabilita asintotica}
q := (- \Delta)^{- 1}\Big({\rm div}\Big( u_\omega \cdot \nabla v + v \cdot \nabla u_\omega + v \cdot \nabla v\Big) \Big)\,. 
\end{equation}
Using that $\| (- \Delta)^{- 1} {\rm div} a \|_{{\cal E}_s} \lesssim \| a \|_{{\cal E}_{s - 1}}$ for any $a \in {\cal E}_s$, one gets the inequality
\begin{equation}\label{stima finale per la pressione}
\begin{aligned}
\| q \|_{{\cal E}_s} & \lesssim \| u_\omega \cdot \nabla v \|_{{\cal E}_{s - 1}}  + \| v \cdot \nabla u_\omega \|_{{\cal E}_{s - 1}} + \| v \cdot \nabla v \|_{{\cal E}_{s - 1}} \,.
\end{aligned}
\end{equation}
Hence arguing as in \eqref{toro bla}, one deduces the estimate $\| q \|_{{\cal E}_s} \lesssim_{s} \delta$. The claimed estimate on $q$ then follows by recalling \eqref{inclusione cal Es} and the proof is concluded (recall that by \eqref{ansatz perturbazione quasi-periodiche}, $v = u - u_\omega$, $q = p - p_\omega$). 
 
\appendix

\section{appendix}

\begin{lemma}\label{appendix0}
Let $n \in \N$, $\zeta > 0$ and $f : [0, + \infty] \to [0, + \infty]$ defined by $f(y) := y^n e^{- \zeta y}$. Then 
$$
{\rm max}_{y \geq 0} f(y) = (n/\zeta)^n e^{- n}
$$
\end{lemma}
\begin{proof}
One has that $f \geq 0$ and 
$$
f(0) = 0, \quad \lim_{y \to + \infty} f(y) = 0\,.
 $$
 Moreover 
 $$
 f'(y) =  y^{n - 1}e^{-  y} \Big( n -  \zeta y\Big)
 $$
 therefore $f$ admits a global maximum at $y = n/\zeta$. This implies that 
 $$
 {\rm max}_{y \geq 0} f(y) = f(n/\zeta) = (n/ \zeta)^n e^{- n}
 $$
 and the lemma follows. 
\end{proof}

\bigskip

\begin{flushright}

\textbf{Riccardo Montalto}

\smallskip

Dipartimento di Matematica ``Federigo Enriques''

Universit\`a degli Studi di Milano

Via Cesare Saldini 50

20133 Milano, Italy

\smallskip

\texttt{riccardo.montalto@unimi.it}
\end{flushright}

\end{document}